\documentclass[a4paper,notitlepage]{article}
\usepackage[applemac]{inputenc}
\usepackage[T1]{fontenc}
\usepackage{amsmath}
\usepackage{amsfonts}
\usepackage{amssymb}
\usepackage{latexsym}
\usepackage{stmaryrd}
\usepackage[english]{babel}
\usepackage[dvips]{graphicx}
\usepackage{longtable}
\usepackage{amsthm}

\usepackage{color}

\usepackage{geometry}

\newtheorem{theo}{Theorem}[section]
\newtheorem{prop}{Proposition}[section]

\newtheorem{lem}{Lemma}[section]
\newcommand{\dx}{\:\text{\textnormal{d}}}
\newcommand{\R}{\mathbb{R}}
\newcommand{\N}{\mathbb{N}}
\newcommand{\Q}{\mathbb{Q}}
\newcommand{\eps}{\varepsilon}
\newcommand{\ve}{\varepsilon}

\newcommand{\pical}{\mathcal{P}}
         
\newcommand{\deb}{\rightharpoonup}

\newcommand{\cvge}[1]{\xrightarrow[#1]{}}

\title{A sharp inequality for transport maps in $W^{1,p}(\R)$ \\ via approximation}
\author{Jean Louet\thanks{Département de Mathématiques, Bât. 425,
    Faculté des Sciences, Université Paris-Sud 11, F-91405 Orsay
    cedex, France (\texttt{jean.louet@math.u-psud.fr}, \texttt{filippo.santambrogio@math.u-psud.fr})}, Filippo Santambrogio\footnotemark[1]}
\date{\today}

\begin{document}
\maketitle

\paragraph{Abstract.} For $f$ convex and increasing, we prove the inequality $ \int f(|U'|)  \geq \int f(nT')$, every time that $U$ is a Sobolev function of one variable and $T$ is the non-decreasing map defined on the same interval with the same image measure as $U$, and the function $n(x)$ takes into account the number of pre-images of $U$ at each point. This may be applied to some variational problems in a mass-transport framework or under volume constraints.



\section{Introduction}

This short paper starts from the following easy question: among maps $U:\Omega\to\Omega'$ with prescribed image measure $\nu$, which is the one with the smallest $H^1$ norm? 

This kind of questions could arise from optimal transport, when two measures $\mu\in\pical(\Omega)$ and $\nu\in\pical(\Omega')$ are fixed, and, instead of considering only costs depending on $(x,U(x))$, we also look at higher-order terms, involving $DU(x)$. It can also arise in incompressible elasticity, where the minimization of the stress tensor (quadratic in $DU$) is standard, and the incompressibility could be expressed through constraints on the image measure rather than by a determinant condition (which is actually equivalent for regular and injective maps). Moreover, in calculus of variations, energy-minimization problems under ``volume constraints'' have already been studied (see \cite{a-a-c,AmbFonMarTar}), and \cite{bu-ri} pointed out that this may be interpreted as a constraint on $\nu$.

If one looks at the $H^1$ norm, it is easy to see that the $\int |U|^2$ part of such a norm does not play any role, since its value is $\int |y|^2d\nu$ and is fixed by the constraint. Hence, we only want to minimize the $L^2$ norm of the derivative part. The easiest case is the 1D one, where we can compare $U$ to any injective function $T$ with the same image, and impose equality of the image measure densities, thus obtaining 
\begin{equation}\label{T'U'}
\frac{1}{|T'(T^{-1}(x))|}=\sum_{y\in U^{-1}(x)}\frac{1}{|U'(y)|}.
\end{equation} 
This shows that the values of $|T'|$ are globally smaller than those of $|U'|$, thus suggesting that the $H^1$ norm of $U$ is larger than that of $T$. If one only uses the pointwise inequality $|U'(y)|\geq |T'(T^{-1}(U(y)))|$, a first proof would give $\int |U'|^2\geq \int n|T'|^2$, where $n$ is a term taking into account the number of points with the same image through $U$ (we will enter into details below). Yet, convexity yields a sharper inequality, namely $\int |U'|^2\geq \int n^2|T'|^2$. This can be generalized to other powers of the derivative, thus getting  $\int |U'|^p\geq \int n^p|T'|^p$.  It can also be furtherly generalized to $\int f(|U'|)\geq \int f(n|T'|)$ for any convex and increasing function $f:\R_+\to\R_+$. 

Similar inequalities could be obtained in higher dimensions, letting the determinant of the Jacobian appear, since the condition \eqref{T'U'} becomes $1/|\det(DT(T^{-1}(x)))|=\sum_{y\in U^{-1}(x)}1/|\det(DU(y))|.$ In such a case one can obtain an inequality like $\int f(|\det DU|)\geq \int f(n(x)|\det DT|)$, where the functional on the gradient part is a polyconvex one (see \cite{polyconvex} for the very first paper on the subject). Yet, there are several difficulties to get this result in higher dimensions, even if one chooses $T=\nabla \phi$ with $\phi$ convex (Brenier's transport \cite{brenier}) to replace increasing map we took in 1D. Actually, how these inequalities are obtained? The first step is to use a change-of-variable technique, the area or co-area formula (since we are in a case where the dimensions of the definition and target spaces agree), and get an equality. Then, condition \eqref{T'U'} (or its multidimensional counterpart with determinants), together with the convexity of $f$, gives the inequality we want. Thus, the coarea formula allows to get the thesis quite rapidly (even if we think that the trick of the convexity inequality on $f$ has not been noticed so far), but an assumption is needed: we actually need to use \eqref{T'U'} and this requires at least $T'\neq 0$. This assumption prevents singular parts in the measure $\nu$ and, since we want to give a general statement for any $\nu$, we then go on by approximation. It is what we do in 1D in this paper, and it seems much more difficult to handle the same strategy in higher dimension. Since we mentioned the arbitrariness of the target measure $\nu$, let us spend few words on the starting measure as well (that we called $\mu$): actually, in this paper we only consider the spatially invariant case, i.e. $\mu$ equal to the Lebesgue measure $\lambda$ and the functional $\int f(|U'|)$ with no explicit dependence on $x$. Actually, it would be interesting to look at arbitrary $\mu$ for optimal transport purposes, but some counter-examples exist when $\mu$ has a non-constant density. We will not enter into details on it here.

On the contrary, let us discuss a while the applications of this inequality and the need for a sharp version. One of the first possibility is to directly apply this inequality in a minimization problem for functionals like $\int f(|U'|)d\mu$, and it proves that the monotone $T$ is optimal. As we said, this is not a general obvious fact and is false when $\mu\neq\lambda$ (the reason being that \eqref{T'U'} would make the values of the density at different points appear). This also implies the optimality of $T$ when the cost is of the form $\int  \left[f(|U'|)+h(|U(x)-x|)\right]d\mu,$ where $h$ is convex as well, since $T$ is known (from the optimal transport theory) to minimize the second part as well. Yet, for general Lagrangian cost functions $\int L(x,U(x),U'(x))dx$ the situation is trickier, and in some of these cases we exactly expect that the sharp version of the inequality could help proving the injectivity of the optimal $U$. Actually, it is possible to prove that every continuous non-injective function $U$ of one variable has a maximal non-injectivity interval $I$ where every image has at least $n\geq 2$ points in its pre-image, and where it is possible to replace $U$ with a ``local'' version of $T$. In the $H^1$ case (i.e. when $L$ contains a quadratic part in $U'$), our estimate implies that this term in the energy is not only larger for $U$ than for $T$, but it also quantifies the gain. Actually, replacing $U$ with $T$ decreases the energy of at least $3\int_I |U'|^2$ (since $n^2\geq 4$), and this could allow to compensate what is lost in the non-gradient part thanks to a Poincaré-type inequality.

\section{Basis for the inequality }

\paragraph{Notations and statement}
We first precise the inequality that we want to prove.
Let $I=[a,b]$ be a segment of $\R$, take $U$, $T \in W^{1,1}(I)$ such that $T_\# \lambda = U_\# \lambda = \nu$, where $\lambda$ is the Lebesgue measure, and choose $T$ non-decreasing. Actually, for any measure $\nu$, there exists a unique non-decreasing map $T$ such that $T_\#\lambda=\nu$, thanks, by the way, to standard results in optimal transportation (see \cite{brenier}); the fact that $T\in W^{1,1}$ is in this section an assumption: in general it depends on lower bounds on the density of $\nu$. 

Let $f : \R^+ \to \R^+$ be a convex and non-decreasing function, such that 
$f \circ |U'|$ is integrable. Our goal is to prove that, if we note, for $x \in I$,
$$ n(x) = \# U^{-1}(T(x)) $$
the number of points of $I$ having $T(x)$ as image by $U$, then $f(n(x)T'(x))$ is integrable and the following inequality holds:
\begin{equation}
 \int_I f(|U'(x)|) \dx x \geq \int_I f(T'(x)n(x)) \dx x \label{ineq}
  \end{equation}
(we use the convention that, in case $n(x)=\infty$ and $T'(x)=0$, we consider $T'(x)n(x)=0$).

\paragraph{Coarea and convexity}
We first give a sketchy idea via the coarea formula. This formula gives
$$\int_a^b g(x)|U'(x)| \dx x = \int_\R \dx y \left(\sum_{x\in U^{-1}(y)}g(x)\right),$$
for every measurable function $g$ (here $U$ is usually supposed to be Lipschitz, but a slightly different but equivalent version exists for $U\in BV$, and thus for $U \in W^{1,1}$, see \cite{e-g}). Let us take $g=f(|U'|)/|U'|$ (we don't really need to deal with the case where $U'$ can vanish, even if this would be easy at least when $f(0)=f'(0)=0$). This gives 
$$\int_a^b f(|U'(x)|) \dx x = \int_\R \dx y \left(\sum_{x\in U^{-1}(y)}f(|U'(x)|)\frac{1}{|U'(x)|}\right).$$
We now use the condition \eqref{T'U'} (we don't precise the assumptions to guarantee its validity and we assume $T$ to be bijective), which gives 
$$\sum_{x\in U^{-1}(y)}\frac{T'(T^{-1}(y))}{|U'(x)|}=1.$$
This, together with the convexity of $f$, yields
$$\int_a^b f(|U'(x)|) \dx x \!\geq \!\int_\R \frac{\dx y}{T'(T^{-1}(y))} f\left(\sum_{x\in U^{-1}(y)}\!\! |U'(x)|\frac{T'(T^{-1}(y))}{|U'(x)|}\right)\!=\!\int_\R \frac{\dx y}{T'(T^{-1}(y))}\!  f\left(T(T^{-1}(y))n(T^{-1}(y))\right),$$
which gives the thesis after a change of variable $x=T^{-1}(y)$.

\paragraph{The piecewise monotone case}
In this paragraph, we make the following assumption: there exists a subdivision $a = x_1 < ... < x_l = b$ of the interval $I$ such that, on each segment $[x_i,x_{i+1}]$, $U$ is of class $C^1$ and its derivative doesn't vanish. This implies that $\nu=U_\#\lambda$ has a density which is piecewise $C^0$ and bounded from below. In particular, the corresponding map $T$ is also piecewise $C^1$ and $T'$ is bounded from below. In such a case we can give a complete proof, with no need to evoke the coarea formula.
\begin{prop}
Under above assumption, inequality \eqref{ineq} is true.
\end{prop}
\begin{proof}
We denote by $y_1 < ... < y_j$ ($j \leq l$) the points of the set $U(\{x_1,...,x_l\})$ and, adding some elements to the set $\{x_1,...,x_l\}$, we can assume that this set coincides with $U^{-1}(\{y_1,...,y_j\})$. Then:
\begin{itemize}
\item for any $1 \leq i \leq l-1$, there exists $1 \leq k \leq j-1$ such that $U : \:]x_i,x_{i+1}[ \: \to \:]y_k,y_{k+1}[$ is bijective, strictly monotone. We denote by $\varphi_i$ its inverse function, and $u_i = \varphi_i \circ T$;
\item for each $1 \leq k \leq j-1$, let us denote by $A_k$ the set of indexes $i$ such that $U(]x_i,x_{i+1}[)$ is precisely the interval $]y_k,y_{k+1}[$, and $n_k= \#A_k$, and notice that the function $y \mapsto \#U^{-1}(\{y\})$ has $n_k$ for constant value on $]y_k,y_{k+1}[$; 
\item the equality \eqref{T'U'} is true almost everywhere, and more precisely :
\end{itemize}

\begin{lem}\label{T'vrai}
Let us denote by $z_k=T^{-1}(y_k)$, $1 \leq k \leq j$. Then, for $x$ in $I \backslash \{z_1,...,z_j\}$, we have
\begin{equation}\label{commonvalue}
 \frac{1}{T'(x)} = \sum_{i \in A_k} \frac{1}{|U'(u_i(x))|} 
 \end{equation}
where $k$ is the index such that $z_k < x < z_{k+1}$. The common value in \eqref{commonvalue} also coincides with the density of the measure $\nu$ at the point $T(x)$ (a density which is thus piecewise continuous on $[c,d]$).\end{lem}

\begin{proof} Let us take $y_k < c < d < y_{k+1}$ and compute the value $\nu([c,d])$. Firstly, we have:

$$ \nu([c,d]) = \lambda(T^{-1}([c,d])) = \int_{T^{-1}(c) < x < T^{-1}(d)} \dx x   = \int_c^d \frac{\dx y}{T'(T^{-1}(y))} $$
where the last equality is obtained by changing of variables $T(x)=y$. On the other hand, $\nu([c,d]) = \lambda(U^{-1}([c,d])$ with $U^{-1}([c,d]) = \bigcup\limits_{i \in A_k} \varphi_i([c,d]) $ and for $i \in A_k$ :
$$\lambda(\varphi_i[c,d]) = ±\int_{\varphi_i(c)}^{\varphi_i(d)} \dx x = \int_c^d \frac{\dx y}{|U'(\phi_i(y))|} =  \int_c^d \frac{\dx y}{|U'(u_i(T^{-1}(y)))|} ;$$
this proves the lemma. $\qedhere$ \end{proof}

Thus, we have the following equalities :
$$\int_a^b f(|U'(x)|) \dx x = \sum\limits_{i=1}^{l-1} \int_{x_i}^{x_{i+1}} f(|U'(x)|) \dx x = \sum\limits_{k=1}^{j-1} \left( \sum\limits_{i \in A_k} \int_{x_i}^{x_{i+1}} f(|U'(x)|) \dx x \right),$$
where the sum over $i$ has $n_k$ terms. On the interval $]x_i,x_{i+1}[$, we set $x = u_i(y)$, and since the derivative of $u_i$ is $ u_i' = T'/(U'\!\circ\! u_i)$, we obtain

\begin{equation} \int_a^b f(|U'|) = \sum\limits_{k=1}^{j-1} \sum\limits_{i \in A_k} \left( \int_{z_k}^{z_{k+1}}  f(|U'(u_i(y))|) \frac{T'(y)}{|U'(u_i(y))|} \dx y \right). \label{eq1} \end{equation}

In the equation \eqref{eq1}, Lemma \ref{T'vrai} gives $ \sum\limits_{i \in A_k} \frac{T'(y)}{|U'(u_i(y))|} = 1  $ and, since $f$ is convex:
$$ \int_a^b f(|U'|) \geq \sum\limits_{k=1}^{j-1} \int_{z_k}^{z_{k+1}} f \left( \sum\limits_{i \in A_k} |U'(u_i(y))| \frac{T'(y)}{|U'(u_i(y))|} \right) \dx y; $$
in the second sum, the terms $|U'(u_i(y))|$ disappear and it only remains $T'(y)$ which occurs $n_k$ times; since $n_k$ is exactly the constant value of $n$ on $]z_k,z_{k+1}[$, we obtain the integral on this interval of $n(y)T'(y)$; and the first sum gives the integral of this function on the full interval $I$, \textit{i.e.}
$$ \int_I f(|U'(y)|) \dx y \geq \int_I f(n(y)T'(y)) \dx y.\qedhere$$
\end{proof}

\section{Approximation}

To handle the general case, we fix a superlinear non-decreasing function $f$ and we first show the following:

\begin{lem} For every $U\in W^{1,1}(I)$ with $\int f(|U'|)<+\infty$ there exists a sequence $(U_k)_{k \in \N}$ in $W^{1,1}(I)$ such that :
\begin{itemize}
\item $U_k \cvge{k} U$ in $W^{1,1}(I)$ ;
\item $f \circ |U_k'| \cvge{k} f \circ |U'|$ in $L^1(I)$ ;
\item for each $k$, $U_k$ is piecewise affine with $U'_k \neq 0$ a.e..
\end{itemize}
\end{lem}

\begin{proof} First notice that if the the thesis is true when replacing $f$ with $x \mapsto f(x)+x$, then it stays true for the original function $f$. This allows to assume that $f'$ is bounded from below by a positive constant; then $f : \R^+ \to \R^+$ is bijective and its inverse function is Lipschitz. Let $(h_k)_k$ be a sequence of positive and piecewise constant functions (say, on dyadic intervals of length $(b-a)/2^k$), such that $h_k \to f \circ |U'|$ in $L^1(I)$. We define $U_k'$ by
$$ U'_k(x) = \text{sgn}(U(x^+_k)-U(x^-_k)) f^{-1}(h_k(x)) $$
where $x_k^+$, $x_k^-$ are the dyadic numbers around $x$ (this is a non-ambiguous definition for a.e. $x \in I$).
\par We have $f \circ |U_k'| \to f \circ |U'|$ in $L^1(I)$ and we want to prove $U_k' \to U'$. First, notice that, since $f^{-1}$ is Lipschitz, we easily get $|U_k'| \to |U'|$ in $L^1$. Moreover, up to subsequences the convergence also holds a.e. on $I$. Thus, it is enough to manage the sign and prove that $U_k' \to U'$ a.e. on $I$. This convergence holds on any non-dyadic point where $U$ is differentiable with $U'\neq 0$ (which imposes the sign of  $U(x^+_k)-U(x^-_k)$). These points, together with those where $|U'|=U'=0$ cover almost all the interval ; this gives $U'_k \to U'$ a.e. on $I$ and, by dominated convergence, we obtain $U_k' \to U'$ in $L^1$.
\par Then, if we take for $U_k$ the primitive of $U_k'$ which has the same value of $U$ at $a$, we obtain a sequence $(U_k)_k$ of piecewise affine functions such that, by construction, $U_k \to U$ in $W^{1,1}(I)$ and $f \circ |U'_k| \to f\circ |U'|$ in $L^1(I)$. 
\end{proof}

In particular, each $U_k$ verifies the condition of the first section, thus the inequality \eqref{ineq} is true with $U_k$, $T_k$ (the non-decreasing function with same image measure) and $n_k = \#U_k^{-1} \circ T_k$.

\paragraph{Remark.} Thanks to the inequality in the piecewise affine case and since $f$ is non-decreasing and superlinear, $n_k \geq 1$ and $\int f(|U'_k|)$ has a limit, we infer that the sequence $(f(T'_k))_k$ is bounded in $L^1(I)$, with $f$ superlinear; thus $(T'_k)_k$ is an equi-integrable family, which implies:
\begin{itemize}
\item the sequence $(T_k)_k$ is equi-continuous, thus it admits, up to subsequences, a uniform limit $T$; this limit is obviously a non-decreasing function;
\item from the strong convergence $U_k\to U$ in $W^{1,1}$ we infer a.e. pointwise convergence, which implies that $(U_k)_\#\lambda\deb U_\#\lambda$; analogously, we have $(T_k)_\#\lambda\deb T_\#\lambda$, which implies $U_\#\lambda=T_\#\lambda$. Hence, the function $T$ is exactly the monotone function corresponding to the original function $U$;
\item the sequence $(T'_k)_k$ is weakly relatively compact in $L^1$, which, together with the uniform convergence $T_k\to T$, gives $T\in W^{1,1}(I)$ and $T_k \deb T$ in $W^{1,1}(I)$.
\end{itemize}

\paragraph{Asymptotics of $n_k(x)$ as $k\to \infty$}
To look at the limits of $n_k$, let us define the function $m$ given by
$$ m(x) = \inf \{ \liminf\limits_{k \to +\infty} n_k(x_k) \;:\; x_k \to x\}.$$
This function is actually the $\Gamma$-$\liminf$ of the functions $n_k$ (see \cite{braides}). A general result on $\Gamma$-$\liminf$ functions gives that $m$ is lower semicontinuous on $I$ (it is easy to check it via a sort of diagonal sequence).
\smallskip

We are interested in the following.
\begin{lem} \label{nk} For almost every $x \in I$ such that $T'(x)\neq 0$, we have
$m(x)\geq n(x)$.
\end{lem}

\begin{proof} Let us first show that this inequality holds if $y=T(x)$ is not a local extremum of $U$; thus, if we take $x' \in I$ and $\delta>0$ with $U(x')=y$, there exist $x^-$, $x^+ \in \:]x'-\delta,x'+\delta[$ with $U(x^-) < U(x') < U(x^+)$. Let $(x_k)_k$ be a sequence of $I$ converging to $ x$, and $y_k = T_k(x_k)$. Thanks to the uniform convergence of $(T_k)_k$ to $T$, $y_k \to y$. Let $p$ be a finite integer such that $p \leq n(x)$; we will show that we can find $p$ distinct points having $T_k(x_k)$ as image by $U$, for $k$ large enough (depending on $p$).
\par Let $z_1 < ... < z_p \in U^{-1}(y)$, and $\delta < \min_j (z_{j+1}-z_j)$. By the assumption on $y$, we can find $\ve>0$ and some points $z_j^+$ and $z_j^-$ in each interval $]z_j-\delta, z_j+\delta[$ such that $U(z_j^-) +\ve< U(z_j) < U(z_j^+)-\ve$.
\par Since $U_k \to U$ pointwisely on $I$, the sequence  $(U_k(z_j^-))_k$ (resp.  $(U_k(z_j^+))_k$, $(U_k(z_j))_k$) converges to $U(z_j^-)$ (resp. $U(z_j^+)$, $U(z_j)$). There exists $k_0 \in \N$ such that, for $k \geq k_0$, we have $U_k(z_j^-) \leq U(z_j^-)+\eps/2 \leq y-\eps/2$ and $U_k(z_j^+) \geq U(z_j^+) - \eps/2 \geq y+\eps/2$; moreover, since $y_k \to y$, we can assume that,  for $k \geq k_0$, $ y-\eps/2 \leq y_k \leq y+\eps/2$ ; combining these two points, we have
$$ \mbox{ for all }\, k \geq k_0, \quad  U_k(z_j^-) \leq y_k \leq U_k(z_j^+); $$
then by the intermediate value theorem, since $U_k$ is continuous, for any $j$, there exist $z_j^k$ between $z_j^+$ and $z_j^-$, such that $U_k(z_j^k) = y_k$. The points $z_j^k$, $1 \leq j \leq p$, are distinct, since they belong to disjoint intervals $]z_j-\delta, z_j+\delta[$. Hence, $n_k(x_k) \geq p$ for $k \geq k_0$. The proof is complete if $T(x)$ is not a local extremum of $U$.
\par The last step consists in showing that the set $A$ of points $x$ such that $T(x)$ is a local extremum for $U$ and verifying furthermore $T'(x) > 0$ is negligible for the Lebesgue measure. Indeed, $y$ is a local maximum of $U$ if, and only if, $y = \max\limits_{J_{q,r}} U$ with $q \in \Q \cap I$, $r \in \Q_+^*$ and $J_{q,r} = ]q-r,q+r[ \cap I$; therefore,

$$ A = \bigcup\limits_{q \in \Q\cap I, r \in \Q_+^*} \left\{ T^{-1} (\max\limits_{J_{q,r}} U) \cap \{T'>0 \},  T^{-1}(\min\limits_{J_{q,r}} U) \cap \{T'>0 \} \right\}$$
This equality proves that $A$ is mesurable, and it is enough to prove that for each level $t$ we have $\lambda \left( T^{-1}(t) \cap \{T'>0 \} \right) = 0 $; this is true since $T'=0$ a.e. on any level set of $T$ (which is an interval).
\end{proof}

\paragraph{Conclusion by semi-continuity}
 We denote for $x \in I$ and $k,j \in \N$:
\begin{gather*}
n_k^j(x) = \min \left( j, \inf\limits_{ y \in I} \{ j|x-y|+n_k(y) \} \right) ;\quad
 h_j(x) = \lim\limits_{k \to +\infty} n_k^j(x) ;
  \end{gather*}
($h_j$ exists since the family $(n_k^j)_k$ is, for each $j$, uniformly
bounded and equi-Lipschitz, thus we can assume that it admits a
uniform limit up to subsequences). Let us notice that for any $j$,
$n_k^j \leq n_k$ (take $x=y$ in the definition of $n_k^j$). Moreover,
by Lemma \ref{nk}, $m \geq n$ on $I$. Let us show the following lemma:

\begin{lem} For any $j \in \N$, we have $h_j \geq m_j$ on $I$, where $m_j$ is defined as
  $$m_j(x) = \min\left( j, \inf\limits_{y \in I} \{j|x-y|+m(y) \} \right).$$
   \end{lem}

\begin{proof} Set $n_k^j = \min(j,\tilde{n}_k^j)$ and $m_j = \min(j,\tilde{m}_j)$, where
$$ \tilde{n}_k^j(x) = \inf\limits_{y \in I} \{ j|x-y|+n_k(y) \} \text{ and } \tilde{m}_j(x) = \inf\limits_{y \in I} \{j|x-y|+m(y) \}.$$
By definition, there is a sequence $(y_k)_k$ such that $\tilde{n}_k^j(x) \leq j|x-y_k|+n_k(y_k) \leq \tilde{n}_k^j(x)+1/k$ for any $k$; taking the minimum with $j$, we obtain
$$ n_k^j(x) \leq \min\left(j, j|x-y_k|+n_k(y_k)\right) \leq n_k^j(x)+\frac{1}{k}.$$
We may assume by compactness that $y_k \to y \in I$, and, by definition of $h_j$, we have $\min(j,j|x-y_k|+n_k(y_k)) \cvge{k} h_j(x)$. Moreover, by definition of $m$, we have $\liminf_{k \to +\infty} n_k(y_k) \geq m(y)$, which gives
$$ \min\left(j,j|x-y|+m(y)\right) \leq h_j(x) $$
for $y=\lim_k y_k \in I$ ; since $m_j(x)$ is the infimum over $y$ of the left-hand side, we obtain $h_j(x) \geq m_j(x)$. 
\end{proof}

The functions $m_j$ that we just introduced are the usual Lipschitz ``regularization'' of the l.s.c. function $m$, and we will use (without proving) the following standard lemma.

\begin{lem} The sequence of functions $(m_j)_{j \in \N}$ is non-decreasing, and has $m$ for pointwise limit. \end{lem}

We now return to our main result:
\begin{theo}
Let $f$ be convex and non-decreasing, $U\in W^{1,1}$ such that $\int f(|U'|)<+\infty$ and $T$ monotone non-decreasing such that $T_\#\lambda=U_\#\lambda$. Then the inequality \eqref{ineq} holds.
\end{theo}
\begin{proof}
First of all suppose that $f$ is superlinerar. We use the approximation defined in this section. Section 2 proves that, for any $k$ :
$$ \int_I f(|U_k'(x)|) \dx x \geq \int_I f(n_k(x)T'_k(x)) \dx x $$
and thanks to the non-decreasing behavior of $f$ and to the remarks about $n_k$, $n_k^j$ and $h_j$, we have the following inequalities, which are true for $k\geq k_0$ for every $\delta>0$ and $j$ ($k_0=k_0(\delta,j)$):
\begin{equation}  \int_I f(n_kT_k') \geq \int_I f(n_k^jT_k') \geq \int_I f((h_j-\delta)T_k') \geq \int_I f((m_j-\delta) T_k'). \label{ineg} \end{equation}
For some fixed $\delta >0$ and $j \in \N$, the functional
$$ T \in W^{1,1} \mapsto \int_I f((m_j(x)-\delta) T'(x)) \dx x  $$
is lower semi-continuous with respect to the weak convergence in $W^{1,1}$ (see \cite{daco}). Thus, taking the limit $k \to +\infty$ in \eqref{ineg} gives
$$ \liminf_{k \to +\infty} \int_I f(n_k T_k') \geq \int_I f((m_j-\delta) T') ;$$
and by monotone convergence, taking the limit $j \to +\infty$ and $\delta \to 0$ in the right-hand side gives
$$ \int_I f((m_j-\delta) T') \cvge{j} \int_I f((m-\delta)T') \xrightarrow[\delta \to 0]{} \int_I f(mT'). $$
Since $m \geq n$ a.e. on the set $\{T' \neq 0 \}$ and $f$ is non-decreasing, the proof is complete for $f$ superlinear.

If $f$ has linear growth, it is sufficient to select a positive, convex, increasing and superlinear function $\tilde f$ such that $\int \tilde f(|U'|)<+\infty$; if we fix $\ve > 0$, $f+\ve\tilde f$ is superlinear and non-decreasing, thus
$$ \mbox{ for all $\ve > 0$ we have }\; \int f(|U'|) + \ve \int \tilde f(|U'|) \geq \int f(nT') + \ve \int \tilde f(nT') \geq \int f(nT') $$
and passing to the limit as $\ve\to 0$ gives the result. 
\end{proof}

\end{document}